\documentclass{article}
\usepackage{graphicx} 
\usepackage{amssymb,amsmath,amsthm,dsfont,enumitem}
\usepackage[english]{babel}
\usepackage[nottoc]{tocbibind}
\usepackage[toc]{appendix}
\usepackage{biblatex}

\theoremstyle{definition}
\newtheorem{dfn}{Definition}[section]
\theoremstyle{remark}

\newtheorem{rmk}[dfn]{Remark}
\theoremstyle{plain}
\newtheorem{prp}[dfn]{Proposition}
\newtheorem{lem}[dfn]{Lemma}
\newtheorem{thm}[dfn]{Theorem}

\title{A fixed point theorem for the action of linear higher rank algebraic groups over local fields on symmetric spaces of infinite dimension and finite rank}
\author{Federico Viola}
\date{October 13, 2026}

\begin{document}

\maketitle

\section*{Abstract}

\textit{Let $\mathbb{F}$ be a non-archimedean local field of characteristic zero whose residue field has at least three elements. Let $G$ be an almost simple linear algebraic group over $\mathbb{F}$, with $\mathrm{rank}_\mathbb{F}(G)\geq 2$. Let $X$ be a simply connected symmetric space of infinite dimension and finite rank, with non-positive curvature operator. We prove that every continuous action by isometries of $G$ on $X$ has a fixed point. If the group $G$ contains $\mathrm{SL}_3(\mathbb{F})$, the result holds without any assumption on the non-archimedean local field $\mathbb{F}$. The result extends to cocompact lattices in $G$ if the cardinality of the residue field of $\mathbb{F}$ is large enough, with a bound that depends on $\mathrm{rank}_\mathbb{F}(G)$.} \\

\section{Introduction}

The theory of finite-dimensional or unitary representations of algebraic groups is a very classical subject. A topic of much recent interest is representations preserving a sesquilinear form of finite index on a Hilbert space, or representations into Pontryagin spaces \cite{pontryagin}. These are particularly interesting because their projective versions appear as isometry groups of non-positively curved symmetric spaces of infinite dimension and finite rank (\cite{duc2023}, Theorem 3.3). \\

In the case of Lie groups, which include linear algebraic groups over $\mathbb{R}$ or $\mathbb{C}$, there are examples of continuous irreducible representations of groups of rank one into Pontryagin spaces of arbitrary index \cite{delzantpy,monodpy}, and in contrast no such representation exists for groups of higher rank (\cite{duc2023}, Theorem 1.1). In the case of linear algebraic groups over non-archimedean local fields, there exist continuous irreducible representations of groups of rank one, such as $\mathrm{SL}_2(\mathbb{Q}_p)$, into the isometry group of the infinite-dimensional real hyperbolic space (\cite{bim}, Theorem C), that is, the projective real Pontryagin space of index one. \\

In this paper, we focus on the case of linear higher rank algebraic groups over non-archimedean local fields. We will prove the following theorem:

\begin{thm} \label{Main}
    Let $\mathbb{F}$ be a non-archimedean local field of characteristic zero whose residue field $k$ has at least three elements. Let $G$ be an almost simple linear algebraic group over $\mathbb{F}$, with $\mathrm{rank}_\mathbb{F}(G)\geq 2$. Let $X$ be an infinite-dimensional simply connected symmetric space with finite rank and non-positive curvature operator. Then, any continuous action of $G$ by isometries on $X$ fixes a point in $X$. \\
\end{thm}

The rank of a symmetric space $X$ is defined as the maximal dimension of an isometrically embedded Euclidean space in $X$. The theorem implies that there are no continuous irreducible representations of $G$ into any Pontryagin space of index $\geq 1$. \\

A first result in this direction was obtained by Arturo Sánchez Gonzáles in his Ph.D. thesis, where he proved that for $n\geq 3$ there are no continuous irreducible representations of $\mathrm{SL}_n(\mathbb{Q}_p)$ to $\mathrm{O}(2,\infty)$ (\cite{arturo}, Teorema C). \\

It is known that any almost simple linear algebraic group $G$ over a non-archimedean local field $\mathbb{F}$ with $\mathrm{rank}_\mathbb{F}(G)\geq 2$ contains a quotient of $\mathrm{SL}_3(\mathbb{F})$ or $\mathrm{Sp}_4(\mathbb{F})$ by a finite subgroup (Proposition I.1.6.2 in \cite{margulis}, which is a consequence of Theorem 7.2 in \cite{boreltits1}). The strategy for the proof of Theorem \ref{Main} will be to restrict ourselves to these two cases and to use a cocompact lattice $\Gamma$ in $G=\mathrm{SL}_3(\mathbb{F})$ or $G=\mathrm{Sp}_4(\mathbb{F})$ to build $\Gamma$-equivariant functions from the Bruhat-Tits building of $G$ to $X$. The existence of a cocompact lattice follows from (\cite{borelharder}, Theorem A) since the field $\mathbb{F}$ has characteristic zero. This is why we need to make that assumption in the theorem: if the field $\mathbb{F}$ has positive characteristic, then a cocompact lattice exists in $\mathrm{SL}_3(\mathbb{F})$ (\cite{borelharder}, Theorem 3.3) but not in $\mathrm{Sp}_4(\mathbb{F})$ (\cite{margulis}, Corollary IX.4.8), therefore, we cannot carry out our proof in that case. However, we are able to solve the case of $\mathrm{SL}_3(\mathbb{F})$ in any characteristic and, as we will see in the proof, this case also does not require the assumption on the residue field. This leads to the following result:

\begin{thm} \label{Main2}
    Let $\mathbb{F}$ be a non-archimedean local field. Let $G$ be an almost simple linear algebraic group over $\mathbb{F}$, containing a quotient of $\mathrm{SL}_3(\mathbb{F})$ by a finite subgroup. Let $X$ be an infinite-dimensional simply connected symmetric space with finite rank and non-positive curvature operator. Then, any continuous action of $G$ by isometries on $X$ fixes a point in $X$. \\
\end{thm}

The study of Hilbert spaces with indefinite sesquilinear forms dates back to Pontryagin \cite{pontryagin}, Iohvidov-Krein \cite{iohkrein}, and Naimark \cite{naimark1,naimark2,naimark3,naimark4}; further treatment can be found in \cite{aziioh,bognar,iohkrlan}. The study of the associated symmetric spaces was first suggested by Gromov (\cite{gromov}, Section 6). The particular case of the infinite-dimensional real hyperbolic space $\mathbf{H}^\infty$ has been treated by Burger, Iozzi, and Monod in \cite{bim}, where they construct equivariant embedding of trees into this space and classify continuous irreducible representations of groups such as $\mathrm{SL}_2(\mathbb{Q}_p)$ (which is a subgroup of the automorphism group of a regular tree) into $\mathrm{Isom}(\mathbf{H}^\infty)$. \\

More recently, simply connected non-positively curved symmetric spaces of infinite dimension and finite rank have been widely studied by Duchesne \cite{ducphd,duc2013,duc2015,duc2023}; his study includes the classification (\cite{duc2015}, Theorem 1.8), the determination of their isometry groups (\cite{duc2023}, Theorem 3.3), and a result about continuous representations of higher rank Lie groups into these isometry groups (\cite{duc2023}, Theorem 1.1), stating that every such representation, provided it has no totally isotropic invariant subspace, splits as a direct sum of finite-dimensional representations and a unitary representation. \\

In another paper \cite{trees}, we have studied representations of tree automorphism groups into Pontryagin spaces, and we have proved that for a wide class of tree automorphism groups, there are no continuous irreducible representations to Pontryagin spaces of index $\geq 2$, meaning that these groups admit no irreducible action on any simply connected non-positively curved irreducible symmetric space of infinite dimension and finite rank $\geq 2$. \\

Turning back to our theorem, it is convenient to define a family of possible target spaces $X$ that includes both finite- and infinite-dimensional spaces. To this end, we define the family $\mathcal{X}$ as follows:

\begin{dfn} \label{familyX}
    The family $\mathcal{X}$ is defined as the class of manifolds $X$ that satisfy the following properties:
    \begin{itemize}
        \item $X$ is either a classical Riemannian manifold or a separable Hilbert manifold.
        \item $X$ is a simply connected symmetric space. This includes both finite- and infinite-dimensional spaces (according to whether $X$ is a classical or a Hilbert manifold).
        \item $X$ has non-positive curvature operator.
        \item There is a finite bound on the dimension of isometrically embedded Euclidean spaces in $X$. This bound is called the \textit{rank} of $X$. \\
    \end{itemize}
\end{dfn}

\begin{rmk}
    In the case where the space $X$ is finite-dimensional, the problem is already well understood. Indeed, let $f: G\rightarrow\mathrm{Isom}(X)$ be any continuous homomorphism. Since $\mathrm{Isom}(X)$ is a Lie group it has no small subgroups, and since $G$ is totally disconnected, by Van Dantzig's theorem there is a basis of neighborhoods of the identity that are compact open subgroups. This implies that there exists a compact open subgroup of $G$ whose image in $\mathrm{Isom}(X)$ is trivial. As $G$ is almost simple, it follows from Theorem I.1.5.6 in \cite{margulis} (main theorem in \cite{tits1}) and Theorem I.2.3.1 in \cite{margulis} (Proposition 6.14 in \cite{boreltits4}) that every normal subgroup of $G$ is finite or cocompact. As $\ker{f}$ is a normal subgroup of $G$ containing a compact open subgroup, it has to be cocompact, meaning that $f$ has compact image. Since $X$ is a CAT(0) space, this implies the existence of a fixed point (\cite{bh}, II.2.8). \\
\end{rmk}

We can start our proof of Theorem \ref{Main} with the reduction:

\begin{prp} \label{unbounded}
    If we can prove Theorem \ref{Main} for $G=\mathrm{SL}_3(\mathbb{F})$ and $G=\mathrm{Sp}_4(\mathbb{F})$, then the general case follows.
\end{prp}

\begin{proof}
    It is shown in \cite{cornulier} (Theorem 1.4, Proposition 1.2) that if $G$ is a simple linear algebraic group over a local field, then every action by isometries of $G$ on any metric space $X$ is either bounded, which means that the orbits are bounded, or metrically proper, which means that for every bounded subset $B\subset X$ the set $\{g\in G: gB\cap B\neq\emptyset\}$ has compact closure. The same holds for almost simple groups (\cite{cornulier}, Lemma 3.1). In case the action is bounded, we can already deduce that it has a fixed point in $X$. In case it is metrically proper, we take a closed subgroup $H<G$ isomorphic to a quotient of $\mathrm{SL}_3(\mathbb{F})$ or $\mathrm{Sp}_4(\mathbb{F})$ by a finite subgroup. We can extend the action of $H$ on $X$ trivially to $\mathrm{SL}_3(\mathbb{F})$ or $\mathrm{Sp}_4(\mathbb{F})$; that action will have a fixed point $x\in X$. Now, the fact that the set $\{g\in G: gx=x\}$ has compact closure provides a contradiction since it contains $H$ which is closed and non-compact. \\
\end{proof}

Proposition \ref{unbounded} allows, in the same way, to restrict Theorem \ref{Main2} to the case $G=\mathrm{SL}_3(\mathbb{F})$. \\

It is enough to show that the action of a cocompact subgroup $\Gamma$ of $G=\mathrm{SL}_3(\mathbb{F})$ or $G=\mathrm{Sp}_4(\mathbb{F})$ is bounded. Our strategy will be to show that for a torsion-free cocompact lattice $\Gamma<G$, that is, a discrete torsion-free cocompact subgroup of $G$. As we have said, the existence of a cocompact lattice $\Gamma'<G$ is guaranteed in our cases, and due to (\cite{garland}, Theorem 2.7), there exists $\Gamma\lhd\Gamma'$ of finite index that is also torsion-free. Moreover, as explained in \cite{garland}, $\Gamma'$ and hence $\Gamma$ is finitely generated. \\

Therefore, in order to prove Theorem \ref{Main} and Theorem \ref{Main2}, it is enough to show the following:

\begin{thm} \label{main}
    Let $\mathbb{F}$ be a non-archimedean local field. Let $G=\mathrm{SL}_3(\mathbb{F})$ or $G=\mathrm{Sp}_4(\mathbb{F})$; in the latter case assume that $\mathbb{F}$ has characteristic zero and its residue field $k$ has at least three elements. Let $\Gamma$ be a torsion-free cocompact lattice of $G$. Let $X\in\mathcal{X}$, as defined in Definition \ref{familyX}. Then any continuous action of $\Gamma$ by isometries on $X$ fixes a point in $X$. \\
\end{thm}

We now present a short outline of the paper. First, we will show (Section 2) that every continuous action of $\Gamma$ by isometries on $X\in\mathcal{X}$ has a fixed point in $\overline{X}=X\cup\partial X$, where $\partial X$ is the boundary at infinity of $X$. This is done by generalizing a result that was proved by Wang in the finite-dimensional case \cite{wang}, using a notion of harmonic functions from the set of vertices in the Bruhat-Tits building of $G$ to $X$. After this, we will show (Section 3) that if a locally compact group $\Gamma$ with Kazhdan's Property (T) has the property that whenever it acts continuously by isometries on a space $X\in\mathcal{X}$ it fixes a point in $\overline{X}$, then the group actually has to fix a point in $X$. This will allow us to conclude our proof. \\

\begin{rmk}
    In \cite{wang}, Wang introduced a notion of admissible weight on simplicial complexes and proved that, if $\Gamma$ is the fundamental group of a finite simplicial complex with an admissible weight satisfying an eigenvalue condition, and $X$ is a complete simply connected finite-dimensional Riemannian manifold of nonpositive sectional curvature (as are the finite-dimensional spaces in our family $\mathcal{X}$), every continuous isometric action of $\Gamma$ on $X$ fixes a point in $\overline{X}$ (\cite{wang}, Theorem 1.1). The list of eligible groups includes cocompact lattices in simply connected simple higher rank algebraic groups over non-archimedean local fields, provided that the cardinality of the residue field is large enough, with a bound that depends on the rank of the group (\cite{wang}, Theorem 4.3). The results in our paper allow to show that, for groups $\Gamma$ that satisfy the assumptions of \cite{wang}, every continuous isometric action on a space $X\in\mathcal{X}$ (including infinite-dimensional ones) has a fixed point in $\overline{X}$. In fact, one can follow the proof of our Theorem \ref{main2}, showing existence of harmonic functions as we do in Section \ref{exhar}, and then show that harmonic functions are constant as in Section 3 of \cite{wang}. Furthermore, if the group $\Gamma$ has Kazhdan's Property (T), as it is the case for lattices in connected almost simple higher rank algebraic groups over local fields (\cite{kazhdanT}, Theorem 1.6.1, Theorem 1.7.1), then one can use our Theorem \ref{main3} to show that $\Gamma$ has to fix a point in $X$. This allows to extend our Theorem \ref{Main} to cocompact lattices in $G$ if the cardinality of the residue field is large enough, with a bound that depends on $\mathrm{rank}_\mathbb{F}(G)$. \\
    \end{rmk}

\section{Fixing a point in $\overline{X}$}

In this section, we prove the following:

\begin{thm} \label{main2}
    Let $\mathbb{F}$ be a non-archimedean local field. Let $G=\mathrm{SL}_3(\mathbb{F})$ or $G=\mathrm{Sp}_4(\mathbb{F})$; in the latter case assume that $\mathbb{F}$ has characteristic zero and its residue field $k$ has at least three elements. Let $\Gamma$ be a torsion-free cocompact lattice of $G$. Let $X\in\mathcal{X}$. Then every continuous action of $\Gamma$ by isometries on $X$ fixes a point in $\overline{X}$. \\
\end{thm}

This result was proved by Wang in \cite{wang} in the case where $X$ is a finite-dimensional Riemannian manifold of non-positive sectional curvature, with an assumption that turns out to be more restrictive on the residue field. The proof of our Theorem \ref{main2} will closely follow the proof in \cite{wang}, with some adaptations: most importantly, we need to take into account the fact that our space $X\in\mathcal{X}$ might be infinite-dimensional. \\

\subsection{The building of $G$ and the eigenvalue property}

First, we introduce the Bruhat-Tits building $\Sigma$ of the group $G$. Bruhat-Tits buildings were first introduced in \cite{bruhattits}; more modern and readable treatments can be found in the books \cite{buildings,garrett}. In particular, we refer to Chapter 6.9 of \cite{buildings} and Chapter 19 of \cite{garrett} for the building of $\mathrm{SL}_3(\mathbb{F})$, and to Chapter 20.1 of \cite{garrett} for the building of $\mathrm{Sp}_4(\mathbb{F})$. We also refer to the more recent papers \cite{buildspn,expspn,zetasp4} for a more detailed treatment of the building of $\mathrm{Sp}_4(\mathbb{F})$. \\

The Bruhat-Tits building $\Sigma$ of $G=\mathrm{SL}_3(\mathbb{F})$ is a Euclidean building of dimension two on which $G$ acts by isometries. It is a union of infinitely many Euclidean planes (apartments), each of them being tessellated by equilateral triangles (chambers). The action of $G$ is transitive on the set of pairs $(\mathcal{A},\mathcal{C})$ where $\mathcal{A}$ is an apartment and $\mathcal{C}$ is a chamber contained in $\mathcal{A}$. The situation is the same in the case $G=\mathrm{Sp}_4(\mathbb{F})$, except that the chambers are right isosceles triangles instead of equilateral. The vertices are divided into two categories: special and non-special, as described in \cite{buildspn,expspn,zetasp4}. Each chamber has two special vertices (at the $45^\circ$ angles) and one non-special vertex (at the $90^\circ$ angle). \\

As $\Sigma$ is a two-dimensional simplicial complex, for $i=0,1,2$ we may call $\Sigma(i)$ the set of $i$-simplices in $\Sigma$. The action of $G$ on $\Sigma(0)$ has three orbits, and each $2$-simplex $\sigma\in\Sigma(2)$ has exactly one vertex in each orbit. \\

For every vertex $v\in\Sigma(0)$, we can define its \textit{link graph} $Lk(v)$. It is a graph whose vertices are the vertices in $\Sigma(0)$ adjacent to $v$ (i.~e.~connected to $v$ by an edge in $\Sigma(1)$), and where two vertices $u,w$ are connected by an edge if and only if there is a chamber in $\Sigma(2)$ with vertices $v,u,w$. We call $Lk(v)(0),Lk(v)(1)$ the set of vertices and edges of $Lk(v)$, respectively. \\

We describe the structure of the graph $Lk(v)$. In the case $G=\mathrm{SL}_3(\mathbb{F})$, as mentioned in \cite{buildings} (6.9), it is the flag complex of $\mathrm{SL}_3(k)$, where $k$ is a finite field: the residue field of $\mathbb{F}$. This means that the vertices are points and lines in the projective plane over $k$, and whenever a point is contained in a line, the corresponding vertices are connected by an edge. If $q$ is the cardinality of the field $k$ (a power of a prime number), there are $2(q^2+q+1)$ vertices in total: $q^2+q+1$ points of $\mathbb{P}^2(k)$ and $q^2+q+1$ lines. Every vertex has degree $q+1$. Using the language of generalized polygons \cite{feithigman}, the graph is a generalized triangle with $q+1$ lines through each point and $q+1$ points on each line. \\

In the case $G=\mathrm{Sp}_4(\mathbb{F})$, an analysis of the neighboring vertices of a vertex $v$ in the building was carried out in \cite{zetasp4} (2.5). If $v$ is special, $Lk(v)$ is isomorphic to the flag complex of $\mathrm{Sp}_4(k)$, where $k$ is the residue field of $\mathbb{F}$. The vertices are points and isotropic lines in $\mathbb{P}^3(k)$, and there are $2(q^3+q^2+q+1)$ vertices in total: $q^3+q^2+q+1$ points and $q^3+q^2+q+1$ isotropic lines. Every vertex has degree $q+1$, and the graph is a generalized quadrangle with $q+1$ lines through each point and $q+1$ points on each line. \\

If $v$ is a non-special vertex in the building of $G=\mathrm{Sp}_4(\mathbb{F})$, the graph $Lk(v)$ is a complete bipartite graph with $q+1$ vertices on each side, or a generalized $2$-gon with $q+1$ points and $q+1$ lines. \\

We want to show a property of the graph $Lk(v)$ that corresponds to the first eigenvalue property in \cite{wang}. Let $v\in\Sigma(0)$ and let $Lk(v)$ be its link graph. Let $N$ be the number of vertices and $M$ be the number of edges in $Lk(v)$. After choosing an orientation for every edge in $Lk(v)(1)$, for every $f: Lk(v)(0)\rightarrow\mathbb{R}$ we can define $df: Lk(v)(1)\rightarrow\mathbb{R}$ by $df((u,w))=f(u)-f(w)$. We let $||f||^2=\sum_{u\in Lk(v)(0)}|f(u)|^2$ and $||df||^2=\sum_{(u,w)\in Lk(v)(1)}|f((u,w))|^2$. We want to establish a lower bound for $||df||^2$ in terms of $||f||^2$, for functions that satisfy $\sum_{u\in Lk(v)(0)}{f(u)}=0$. \\

Let $A$ be the adjacency matrix of $Lk(v)$. More explicitly, it is a matrix with $M$ rows and $N$ columns where, after an ordering of the vertices and of the edges has been chosen, the coefficient $A_{ij}$ is $1$ if the $i$-th edge starts from the $j$-th vertex, $-1$ if the $i$-th edge terminates at the $j$-th vertex, and $0$ in all other cases. Then we can consider the function $f$ as a vector in $\mathbb{R}^N$ and the function $df$ as a vector in $\mathbb{R}^M$, so that we have $df=Af$. Therefore, we have $$||df||^2=\langle df,df\rangle=\langle Af,Af\rangle=\langle f,A^TAf\rangle,$$ where we use the standard scalar product in $\mathbb{R}^M$ and in $\mathbb{R}^N$. \\

In order to bound $\langle A^TAf,f\rangle$ in terms of $\langle f,f\rangle=||f||^2$, we need to study the eigenvalues of the matrix $B=A^TA$. It is a $N\times N$ square matrix where the coefficient $B_{ij}$ is $q+1$ if $i=j$, $-1$ if $i\neq j$ and the $i$-th and $j$-th vertex are connected by an edge, $0$ if $i\neq j$ and the $i$-th and $j$-th vertex are not connected. \\

The calculation of minimal polynomials of the incidence matrices of generalized polygons has been carried out by Feit and Higman \cite{feithigman}. We also refer to \cite{garland}, Proposition 7.10, for the complete list of eigenvalues in an equivalent setting. In our setting, the smallest non-zero eigenvalue of the matrix $B$ is:

\begin{itemize}
    \item $q+1-\sqrt{q}$, if $v$ is a vertex in the building of $\mathrm{SL}_3(\mathbb{F})$;
    \item $q+1-\sqrt{2q}$, if $v$ is a special vertex in the building of $\mathrm{Sp}_4(\mathbb{F})$;
    \item $q+1$, if $v$ is a non-special vertex in the building of $\mathrm{Sp}_4(\mathbb{F})$.
\end{itemize}

We may call $\lambda(v)$ the smallest non-zero eigenvalue of the matrix $B$ built from the link graph of the vertex $v$, as in the above list. As it is easy to see that the eigenspace corresponding to the eigenvalue $0$ is the space of constant functions on $Lk(v)(0)$, the following proposition holds:

\begin{prp} \label{eigen}
    If $f: Lk(v)(0)\rightarrow\mathbb{R}$ satisfies $\sum_{u\in Lk(v)(0)}{f(u)}=0$, then $||df||^2\geq\lambda(v)||f||^2$. \\
\end{prp}

\begin{rmk} \label{eigen2}
    The function $f$ in Proposition \ref{eigen} takes values in $\mathbb{R}$, but the result also holds when it takes values in a general Hilbert space. To see this, it suffices to restrict to the finite-dimensional subspace spanned by the (finitely many) values of $f$, then apply the lemma component-wise. \\
\end{rmk}

\subsection{Existence of harmonic functions} \label{exhar}

We now define the most important tool for the proof of Theorem \ref{main2}: harmonic functions on $\Sigma(0)$, which are $\Gamma$-equivariant functions from $\Sigma(0)$ to $X$ that minimize an energy functional. \\

An approach with harmonic functions was also used in \cite{duc2015sup}, where Duchesne proved a superrigidity result for cocompact lattices in higher rank semisimple Lie groups acting on spaces in the family $\mathcal{X}$. \\

Recall that $\Gamma$ is a discrete, torsion-free, cocompact subgroup of $G$. Its action on $\Sigma$ is free and proper (\cite{garland}, Lemma 2.6). Therefore, $\Gamma$ is the fundamental group of the finite simplicial complex $\sigma=\Sigma/\Gamma$, and $\Gamma$-equivariant functions $f: \Sigma(0)\rightarrow X$ correspond to functions $f_\sigma: \sigma(0)\rightarrow X$, where $\sigma(0)$ is the set of vertices in $\sigma$. The correspondence is defined by choosing a fundamental domain $D$ of $\Gamma$ in $\Sigma(0)$ and assigning to $f_\sigma$ the values that $f$ takes at the vertices in $D$. Since $f$ is $\Gamma$-equivariant, all its values are uniquely determined by those at the vertices in $D$ and hence by $f_\sigma$. \\

\begin{dfn}{(Energy functional)} \label{energy}
    The \textit{energy functional} of a $\Gamma$-equivariant function $f: \Sigma(0)\rightarrow X$ is defined by $$E(f):=\sum_{(u,w)\in\sigma(1)}{d(f(\widetilde{u}),f(\widetilde{w}))^2},$$ where $\sigma(1)$ is the set of edges in the $1$-skeleton of $\sigma$, and $\widetilde{u},\widetilde{w}\in\Sigma(0)$ are vertices connected by an edge in $\Sigma(1)$ that project to $u,w\in\sigma(0)$. Since $f$ is $\Gamma$-equivariant and the action of $\Gamma$ on $X$ is by isometries, $d(f(\widetilde{u}),f(\widetilde{w}))$ does not depend on the choice of $\widetilde{u},\widetilde{w}$. \\
\end{dfn}

\begin{dfn}{(Harmonic function)}
    We say that the function $f$ is \textit{harmonic} if it minimizes the energy functional among all $\Gamma$-equivariant functions $\Sigma(0)\rightarrow X$. \\
\end{dfn}

The first step of our proof of Theorem \ref{main2} is the following, which we prove in the rest of this subsection:

\begin{prp} \label{existence}
    If the action of $\Gamma$ does not fix any point in the visual boundary $\partial X$, then there exists at least one harmonic function $f: \Sigma(0)\rightarrow X$. \\
\end{prp}

Since a $\Gamma$-equivariant function $f$ is determined by its values on the fundamental domain $D$ for the $\Gamma$-action on $\Sigma(0)$, we can consider the energy functional as a continuous function $E: X^n\rightarrow\mathbb{R}_{\geq 0}$, where $n$ is the cardinality of $D$. We will denote by $v_1,\ldots,v_n$ the vertices of $D$. \\

The function $E$ is bounded below by zero, and Proposition \ref{existence} is equivalent to $E$ having a minimum. Therefore, if $I\geq 0$ is the infimum of $E$ in $X^n$, we need to show that $E(\textbf{x})=I$ for some $\textbf{x}\in X^n$. \\

We are going to use a notion of dimension at large scale for our spaces $X\in\mathcal{X}$: the notion of telescopic dimension, which was defined in \cite{caply} as the supremum of geometric dimensions of asymptotic cones built from $X$. The notion of geometric dimension was defined inductively in \cite{kleiner} by setting the dimension of a discrete space to be $0$ and defining $\dim(X)=\sup\{\dim(S_x X)+1 | x\in X\}$, where $S_x X$ denotes the space of directions at the point $x$; this definition coincides with the supremum of topological dimensions of compact subsets in $X$. \\

It is proved in \cite{duc2015} (Corollary 1.10) that any space $X\in\mathcal{X}$ has finite telescopic dimension. As the product of CAT(0) spaces is still CAT(0), so is the space $X^n$; moreover, since the space $X$ has finite telescopic dimension and this property is preserved by taking products (\cite{caply}, Lemma A.12), the space $X^n$ has finite telescopic dimension. \\

For all $m\in\mathbb{N}$, let $S_m:=\{\textbf{x}\in X^n:\ E(\textbf{x})\leq I+1/m\}$. Since the energy functional is convex (the distance function in a CAT(0) space is convex), the sets $(S_m)_{m\in\mathbb{N}}$ are a descending chain of closed convex subsets of $X^n$. It is proved in \cite{caply} (Theorem 1.1) that in this case either the intersection $\bigcap_{m\in\mathbb{N}}{S_m}$ is non-empty or the intersection of the visual boundaries $\bigcap_{m\in\mathbb{N}}{\partial S_m}$ is a non-empty subset of $\partial(X^n)$ with intrinsic radius $\leq\pi/2$. \\

If the intersection $\bigcap_{m\in\mathbb{N}}{S_m}$ is non-empty, any $x$ in this intersection must satisfy $E(x)=I$ and hence we are done. Therefore, from now on, we can assume that $\bigcap_{m\in\mathbb{N}}{\partial S_m}$ is non-empty with intrinsic radius $\leq\pi/2$. \\

The action of $\Gamma$ on $X$ extends diagonally to an action on $X^n$. If we show that there exists $\xi\in\partial(X^n)$ fixed by this action, then this implies that $\Gamma$ fixes a point in $\partial X$, since the boundary $\partial(X^n)$ is the spherical join of the boundaries of the factors (\cite{bh}, II.9) and therefore $\xi$ corresponds to a $n$-tuple $(\xi_1,\ldots,\xi_n)\in(\partial X)^n$ together with weights $(\lambda_1,\ldots,\lambda_n)$ satisfying $\sum_{i=1}^{n}{\lambda_i^2}=1$, and $\xi_i\in\partial X$ must be fixed by $\Gamma$ whenever $\lambda_i\neq 0$. Therefore, to complete the proof of Proposition \ref{existence}, it suffices to show that there exists $\xi\in\partial(X^n)$ fixed by the action of $\Gamma$. We will prove the following stronger statement:

\begin{prp}
    If the intersection $\bigcap_{m\in\mathbb{N}}{\partial S_m}$ is non-empty, then every $\xi\in\bigcap_{m\in\mathbb{N}}{\partial S_m}$ is fixed by the action of $\Gamma$.
\end{prp}

\begin{proof}
    Let $\xi\in\bigcap_{m\in\mathbb{N}}{\partial S_m}\subseteq\partial(X^n)$. We can consider a sequence $(\textbf{x}_m)_{m\in\mathbb{N}}$ of points in $X^n$ converging to $\xi$, such that $\textbf{x}_m\in S_m$ for all $m$. To show that $\xi$ is fixed by $\Gamma$, we will show that the sequence $(g\textbf{x}_m)_{m\in\mathbb{N}}$ converges to $\xi$ for all $g\in\Gamma$. \\

    Recall that $v_1,\ldots,v_n$ denote the vertices of the fundamental domain $D$ for the action of $\Gamma$ on $\Sigma(0)$. Each $\textbf{x}_m$ is a $n$-tuple $(\textbf{x}_m^{(1)},\cdots,\textbf{x}_m^{(n)})$ of points in $X$, associated with the unique $\Gamma$-equivariant function $f_{\textbf{x}_m}: \Sigma(0)\rightarrow X$ that satisfies $f_{\textbf{x}_m}(v_i)=\textbf{x}_m^{(i)}$ for $1\leq i\leq n$. \\
    
    Since $\textbf{x}_m\in S_m$, we have $E(f_{\textbf{x}_m})\leq I+1/m$, and in particular $E(f_{\textbf{x}_m})\leq I+1$ for all $m\geq 1$. It follows from the definition of the energy functional that for every $m\geq 1$ we have $d(f_{\textbf{x}_m}(u),f_{\textbf{x}_m}(w))\leq\sqrt{I+1}$ for every $u,w\in\Sigma(0)$ that are linked by an edge in $\Sigma(1)$. From this it follows that, for every $m\geq 1$, we have $$d(f_{\textbf{x}_m}(u),f_{\textbf{x}_m}(w))\leq d_\Sigma(u,w)\sqrt{I+1}$$ for all $u,w\in\Sigma(0)$, where $d_\Sigma(u,w)$ denotes the simplicial distance between $u$ and $w$ in the $1$-skeleton of $\Sigma$, i.~e.~the smallest $k\in\mathbb{N}$ such that there exist $u_0,\ldots,u_k\in\Sigma(0)$ with $u_0=u$, $u_k=w$, and $u_i,u_{i+1}$ linked by an edge in $\Sigma(1)$ for $i=0,\ldots,k-1$. \\

    Fix $g\in\Gamma$ and $m\geq 1$. For $1\leq i\leq n$ we have, since $f_{\textbf{x}_m}$ is $\Gamma$-equivariant: $$d(gf_{\textbf{x}_m}(v_i),f_{\textbf{x}_m}(v_i))=d(f_{\textbf{x}_m}(gv_i),f_{\textbf{x}_m}(v_i))\leq d_\Sigma(gv_i,v_i)\sqrt{I+1}.$$

    Therefore, summing over $i$: $$\sum_{i=1}^{n}{d(gf_{\textbf{x}_m}(v_i),f_{\textbf{x}_m}(v_i))}\leq\sum_{i=1}^{n}{d_\Sigma(gv_i,v_i)\sqrt{I+1}}\leq n\max_{1\leq i\leq n}{d_\Sigma(gv_i,v_i)}\sqrt{I+1}.$$

    As the point $\textbf{x}_m\in X^n$ has coordinates $(\textbf{x}_m^{(1)},\cdots,\textbf{x}_m^{(n)})=(f_{\textbf{x}_m}(v_1),\cdots,f_{\textbf{x}_m}(v_n))$, and the action of $\Gamma$ on $X^n$ is the diagonal action on all individual coordinates, we have $g\textbf{x}_m=(g\textbf{x}_m^{(1)},\cdots,g\textbf{x}_m^{(n)})=(gf_{\textbf{x}_m}(v_1),\cdots,gf_{\textbf{x}_m}(v_n))$ for $g\in\Gamma$. The previous computation gives $$\sum_{i=1}^{n}{d(g\textbf{x}_m^{(i)},\textbf{x}_m^{(i)})}\leq n\max_{1\leq i\leq n}{d_\Sigma(gv_i,v_i)}\sqrt{I+1},$$ which implies that the distance (in $X^n$) between $g\textbf{x}_m$ and $\textbf{x}_m$ is bounded by a constant that does not depend on $m$, and therefore, since the sequence $(\textbf{x}_m)_{m\in\mathbb{N}}$ converges to the point at infinity $\xi$, the sequence $(g\textbf{x}_m)_{m\in\mathbb{N}}$ also converges to $\xi$. Since this is true for every $g\in\Gamma$, the point $\xi\in\partial(X^n)$ is fixed by the action of $\Gamma$. \\
\end{proof}

\subsection{Harmonic functions are constant}

Our next goal is to show the following:

\begin{prp} \label{constant}
    Let $f: \Sigma(0)\rightarrow X$ be harmonic. Then $f$ is constant.
\end{prp}

Together with Proposition \ref{existence}, this will prove Theorem \ref{main2}: given a continuous action of $\Gamma$ by isometries on $X$, if it does not fix any point in $\partial X$ then Proposition \ref{existence} applies and we have a harmonic function $f: \Sigma(0)\rightarrow X$ that must be constant by Proposition \ref{constant}, which means that the whole of $\Sigma(0)$ is sent to a single point $x\in X$ and therefore $x$ is fixed by $\Gamma$. \\

Therefore, we are left to prove Proposition \ref{constant}. In the case $G=\mathrm{SL}_3(\mathbb{F})$, the proof is basically the same as in \cite{wang}; we include it here in a form that is consistent with our situation and notation. In the case $G=\mathrm{Sp}_4(\mathbb{F})$, a small addition will be needed. \\

\begin{dfn}{(Differential)}
    The \textit{differential} of $f$ at $v\in\Sigma(0)$ is defined to be $$Df|_v: Lk(v)(0)\rightarrow T_{f(v)}X$$ such that $\mathrm{exp}(Df|_v(u))=f(u)$ for all $u\in Lk(v)(0)$, where $\mathrm{exp}$ is the exponential map on $T_{f(v)}X$. \\
\end{dfn}

Recall that a $\Gamma$-equivariant function $f: \Sigma(0)\rightarrow X$ is uniquely determined by its values at the vertices $v_1,\ldots,v_n$ in the fundamental domain $D$. In order for the map to be harmonic, the variation of the energy caused by any infinitesimal variation of $f(v_i)$, $1\leq i\leq n$, must be zero. This is expressed in the following lemma:

\begin{lem} \label{sum0}
    Let $f$ be harmonic. Then $$\sum_{u\in Lk(v_i)(0)}{Df|_{v_i}(u)}=0$$ for $1\leq i\leq n$.
\end{lem}

\begin{proof}
    We consider the variation at $v_i$ given by $y\in T_{f(v_i)}X$. We denote by $E(t)$ the energy of the $\Gamma$-equivariant function that sends $v_i$ to $\mathrm{exp}(ty)$ and agrees with $f$ on $D\setminus\{v_i\}$. We have the following:
    \begin{equation*}
        \frac{d}{dt}E(t)=\frac{d}{dt}\left(\sum_{u\in Lk(v_i)(0)}{d(\mathrm{exp}(ty),f(u))^2}\right)=\sum_{u\in Lk(v_i)(0)}{\langle y,-2Df|_{v_i}(u)\rangle},
    \end{equation*}
    where the last equality holds because $f(v_i)$ and $f(u)$ are connected by a minimal geodesic in $X$ and the gradient of the function $x\mapsto d(x,u)^2$ at $x=v_i$ is $-2Df|_{v_i}(u)$. \\
    
    Since $f$ is harmonic it is a critical point of the energy functional, hence $$\sum_{u\in Lk(v_i)(0)}{\langle y,Df|_{v_i}(u)\rangle}=0.$$
    
    Since $y$ was arbitrary in $T_{f(v_i)}X$, the lemma follows. \\
\end{proof}

We can now proceed with the proof of Proposition \ref{constant}. First, we consider the case $G=\mathrm{SL}_3(\mathbb{F})$, where Proposition \ref{eigen} holds for $\lambda(v)=q+1-\sqrt{q}$ that does not depend on $v$. \\

Let $f$ be harmonic. Let $i\in\{1,\ldots,n\}$ and let $u,w\in Lk(v_i)(0)$. Since $X$ is non-positively curved, we have $$||Df|_{v_i}(u)-Df|_{v_i}(w)||\leq d(f(u),f(w)).$$

Therefore, $$\sum_{(u,w)\in Lk(v_i)(1)}{||Df|_{v_i}(u)-Df|_{v_i}(w)||^2}\leq\sum_{(u,w)\in Lk(v_i)(1)}{d(f(u),f(w))^2}.$$

Denoting the left-hand side by $||d(Df)_{v_i}||^2$ and summing over $i$, we obtain
\begin{equation} \label{eq1}
    \sum_{i=1}^{n}{||d(Df)_{v_i}||^2}\leq\sum_{i=1}^{n}{\sum_{(u,w)\in Lk(v_i)(1)}{d(f(u),f(w))^2}}.
\end{equation}

We claim that the right-hand side is equal to $(q+1)E(f)$. Using the notation of Definition \ref{energy}, we see that for every edge $(u,w)\in\sigma(1)$ we are summing $d(f(\widetilde{u}),f(\widetilde{w}))^2$ once for every vertex $v\in\sigma(0)$ such that $u,w,v$ are vertices of a simplex in $\sigma(2)$. The number of such vertices is always $q+1$, since it is the same as the degree of the vertex $w$ in the link graph $Lk(u)$. Therefore, \eqref{eq1} becomes:

\begin{equation} \label{eq2}
    \sum_{i=1}^{n}{||d(Df)_{v_i}||^2}\leq (q+1)E(f).
\end{equation}

Consider now, for $1\leq i\leq n$, the function $g_i=Df|_{v_i}$. It is a function $g_i: Lk(v_i)(0)\rightarrow T_{f(v_i)}X$, and we know by Lemma \ref{sum0} that $\sum_{u\in Lk(v_i)(0)}{g_i(u)}=0$. We can apply Remark \ref{eigen2} and get $||dg_i||^2\geq\lambda(v)||g_i||^2$, where $\lambda(v)=q+1-\sqrt{q}$ does not depend on $v$ and we may call it $\lambda$. This means that $$\sum_{(u,w)\in Lk(v_i)(1)}{||Df|_{v_i}(u)-Df|_{v_i}(w)||^2}\geq\lambda\sum_{u\in Lk(v_i)(0)}{||Df|_{v_i}(u)||^2}.$$

Since $||Df|_{v_i}(u)||^2=d(f(v_i),f(u))^2$, we see that if we sum over $i$, we get $2E(f)$ on the right-hand side, since for each edge $(v,u)\in\sigma(1)$ we are summing the quantity $d(f(\widetilde{v}),f(\widetilde{u}))^2$ twice. Therefore, we have:

\begin{equation} \label{eq3}
    \sum_{i=1}^{n}{||d(Df)_{v_i}||^2}\geq 2\lambda E(f).
\end{equation}

From \eqref{eq2} and \eqref{eq3} we get $(q+1)E(f)\geq 2\lambda E(f)$. We have $\lambda=q+1-\sqrt{q}>\frac{q+1}{2}$ for $q\geq 2$ and since $E(f)\geq 0$, this implies $E(f)=0$. By the definition of $E(f)$, this means that $f$ is constant on $\Sigma(0)$, proving Proposition \ref{constant} in the case $G=\mathrm{SL}_3(\mathbb{F})$. \\

We are left with the case $G=\mathrm{Sp}_4(\mathbb{F})$. Here $\lambda(v)$ is not the same for every $v$, but depends on whether the vertex $v$ is special or not. However, we can repeat the same proof and sum separately over special and non-special vertices. Let $\{v_1,\ldots,v_n\}=S\sqcup NS$ be the partition of $D=\{v_1,\ldots,v_n\}$ into special and non-special vertices, respectively. With Definition \ref{energy} in mind, we can set $E(f)=E_1(f)+E_2(f)$, where $E_1(f)$ is defined by taking the sum only over edges between two special vertices, and $E_2(f)$ is defined by taking the sum only over edges between a special and a non-special vertex (recall that there is no edge between two non-special vertices). \\

We can obtain inequality \eqref{eq1} separately for the sum over $S$ and for the sum over $NS$. Since every chamber in the building has one non-special and two special vertices, inequality \eqref{eq2} becomes:

\begin{equation*}
    \sum_{v\in S}{||d(Df)_v||^2}\leq (q+1)E_2(f),
\end{equation*}

\begin{equation*}
    \sum_{v\in NS}{||d(Df)_v||^2}\leq (q+1)E_1(f).
\end{equation*}

We can continue the proof as above, again summing separately over $S$ and over $NS$. Since $\lambda(v)=q+1-\sqrt{2q}$ for all $v\in S$ and $\lambda(v)=q+1$ for all $v\in NS$, inequality \eqref{eq3} becomes:

\begin{equation*}
    \sum_{v\in S}{||d(Df)_v||^2}\geq(q+1-\sqrt{2q})(2E_1(f)+E_2(f)),
\end{equation*}

\begin{equation*}
    \sum_{v\in NS}{||d(Df)_v||^2}\geq(q+1)E_2(f).
\end{equation*}

Putting this together with the previous inequalities, we get $E_1(f)\geq E_2(f)$ from the sum over $NS$, and then $(q+1)E_2(f)\geq 3(q+1-\sqrt{2q})E_2(f)$ from the sum over $S$. If $E_2(f)>0$, this is only possible if $\frac{\sqrt{2q}}{q+1}\geq\frac{2}{3}$, which implies $q\leq 2$. This means that if $q\geq 3$, we must have $E_2(f)=0$ and therefore $f$ constant, completing the proof of Proposition \ref{constant}. As explained above, Theorem \ref{main2} follows. \\

\begin{rmk}
    The last part of the proof of Proposition \ref{constant} is the reason why in Theorem \ref{Main} we need to assume that the residue field of $\mathbb{F}$ has at least three elements. In Theorem \ref{Main2} we do not need this assumption, since in the case $G=\mathrm{SL}_3(\mathbb{F})$ the inequality works for all $q\geq 2$. \\
\end{rmk}

\section{Obtaining a fixed point inside $X$}

In this section, we complete the proof of Theorem \ref{main}. We know from Theorem \ref{main2} that the action of the torsion-free cocompact lattice $\Gamma<G$ on any $X\in\mathcal{X}$ has to fix a point in $\overline{X}$. Following the notation in \cite{wang}, but emphasizing the dependence on the family $\mathcal{X}$, we name this property ($\text{F}_{\mathcal{X}}$).

\begin{dfn}{(Property ($\text{F}_{\mathcal{X}}$))}
    A topological group $\Gamma$ is said to have Property ($\text{F}_{\mathcal{X}}$) if any continuous action by isometries of $\Gamma$ on any space $X\in\mathcal{X}$ has a fixed point in $\overline{X}$. \\
\end{dfn}

We also use the following definition:

\begin{dfn}{(Property (FH))}
    A locally compact topological group $\Gamma$ is said to have Property (FH) if any continuous affine isometric action of $\Gamma$ on a Hilbert space has a fixed point. \\
\end{dfn}

\begin{rmk}
    If $\Gamma$ is $\sigma$-compact, Property (FH) is equivalent to Kazhdan's Property (T). This is known as the Delorme-Guichardet Theorem (\cite{kazhdanT}, Theorem 2.12.4). \\
\end{rmk}

In this section, we show the following general result, which will allow us to conclude the proof of Theorem \ref{main}:

\begin{thm} \label{main3}
    Let $\Gamma$ be a locally compact group. Assume that $\Gamma$ has Property ($\text{F}_{\mathcal{X}}$) and Property (FH). Then any continuous action by isometries of $\Gamma$ on a space $X\in\mathcal{X}$ has a fixed point in $X$. \\
\end{thm}

\begin{rmk} \label{fiffh}
    If $\Gamma$ has Property ($\text{F}_{\mathcal{X}}$) or Property (FH), every closed finite index subgroup $\Gamma'<\Gamma$ has the corresponding property.
\end{rmk}

\begin{proof}[Proof of Remark \ref{fiffh}]
    For Property (FH), this is (\cite{kazhdanT}, Proposition 2.5.7). For Property ($\text{F}_{\mathcal{X}}$), the proof is similar. If $\Gamma'$ has a continuous action $\alpha$ by isometries on a space $X\in\mathcal{X}$, one can define the induced action $\mathrm{Ind}_{\Gamma'}^{\Gamma}{\alpha}$ on the space $\Xi$ of functions $\xi: \Gamma\rightarrow X$ such that $\xi(gh)=\alpha(h^{-1})\xi(g)\ \forall h\in\Gamma'$, by $$\left(\mathrm{Ind}_{\Gamma'}^{\Gamma}{\alpha}(g)\xi\right)(x)=\xi(g^{-1}x),\ g,x\in\Gamma,\ \xi\in\Xi.$$ This construction is a special case of the one that appears in (\cite{splitting}, 5.4). \\
    
    Let $n=[\Gamma:\Gamma']$. We can identify the space $\Xi$ with $X^n$ by choosing coset representatives, and it is easy to check that the action $\mathrm{Ind}_{\Gamma'}^{\Gamma}{\alpha}$ is continuous and isometric for the $L^2$ product metric. We have $\Xi\in\mathcal{X}$, therefore the action fixes a point in $\xi\in\Xi$ or $\eta\in\partial\Xi$. In the first case, it is immediately checked that $\xi(e)\in X$ is fixed by $\Gamma'$. In the second case, we know that $\partial\Xi$ is the spherical join of $n$ copies of $\partial X$ (\cite{bh}, II.9), that is, the space of $n$-tuples $(\zeta_1,\ldots,\zeta_n)\in(\partial X)^n$ together with weights $(\lambda_1,\ldots,\lambda_n)$ satisfying $\sum_{i=1}^{n}{\lambda_i^2}=1$. It is immediately checked that $\eta$ must correspond to a $n$-tuple of the form $(\zeta,\ldots,\zeta)$ with equal weights, with $\zeta\in\partial X$ fixed by $\Gamma'$. \\
\end{proof}

\subsection{More about the family $\mathcal{X}$} \label{morefamily}

As a first step towards the proof of Theorem \ref{main3}, we determine more explicitly the family of possible spaces $X$. On one hand, it contains all simply connected finite-dimensional symmetric spaces of non-positive curvature. On the other hand, it is shown in \cite{duc2015} (Corollary 1.10) that any space in $\mathcal{X}$, after possibly removing the Euclidean de Rham factor, is isometric to a finite product of irreducible (finite-dimensional) symmetric spaces of non-compact type and spaces of the form $\mathrm{O}(p,\infty)/\mathrm{O}(p)\times\mathrm{O}(\infty)$, $\mathrm{U}(p,\infty)/\mathrm{U}(p)\times\mathrm{U}(\infty)$, $\mathrm{Sp}(p,\infty)/\mathrm{Sp}(p)\times\mathrm{Sp}(\infty)$ with $p\geq 1$, up to homothety. Such spaces were studied in \cite{duc2013}. We use the notation $X_p(\mathbb{K})$ for these spaces, where $\mathbb{K}$ is the field of real, complex, or quaternionic numbers. \\

A general $X\in\mathcal{X}$ is then a product

\setcounter{equation}{0}
\begin{equation} \label{dec}
    X\cong \mathbb{R}^r\times\prod_{i=1}^{h}{M_i}\times\prod_{i=1}^{k}{N_i}
\end{equation}

where each $M_i$ is a simply connected finite-dimensional symmetric space with non-positive curvature and each $N_i$ is homothetic to some $X_p(\mathbb{K})$. \\

We can reduce our problem to each of these single factors. According to a result by Foertsch and Lytchak \cite{foly}, every geodesic space $X$ of finite affine rank has a unique decomposition $$X=Y_0\times\prod_{i=1}^{n}{Y_i}$$ where $Y_0$ is a Euclidean space (possibly a point) and each $Y_i$, $1\leq i\leq n$, is an irreducible metric space not isometric to the real line or to a point. Moreover, each isometry of $X$ is a product of isometries of the single $Y_i$, up to possibly permuting factors which are isometric. This result is a generalization of the classical De Rham decomposition for finite-dimensional Riemannian manifolds \cite{derham}. \\

We will prove in the Appendix that the spaces $X_p(\mathbb{K})$ are irreducible metric spaces. Since they are infinite-dimensional, they cannot appear as factors of the finite-dimensional symmetric spaces ($\mathbb{R}^r$ or $M_i$) in the decomposition \eqref{dec} of $X$. Therefore, the result of Foertsch and Lytchak implies that any isometry of $X$ is a product of an isometry of $\mathbb{R}^r\times\prod_{i=0}^{h}{M_i}$ and an isometry of each factor $N_i$, $1\leq i\leq k$, up to possibly permuting isometric factors. The classical result of De Rham then allows us to conclude that any isometry is a product of isometries of each individual factor in the decomposition, up to possibly permuting isometric factors. \\

There is also an alternative way to see that the factors that appear in the decomposition \eqref{dec} of $X$ coincide with the factors of Foertsch-Lytchak and therefore any isometry of $X$ is a product of isometries of the individual factors up to possibly permuting isometric ones. Indeed, since $X$ is a geodesic space and all its geodesics extend indefinitely in both directions, the same must be true for the factors $(Y_i)_{0\leq i\leq n}$ in the Foertsch-Lytchak decomposition (this follows from the fact that geodesics in a product of metric spaces project to geodesics in factors, as seen in \cite{bh}, I.5.3). Then, as a consequence of Proposition 2.2 in \cite{duc2023}, these factors must be totally geodesic submanifolds of $X$ and therefore coincide (up to the order) with the factors in the decomposition \eqref{dec}. This was suggested by Bruno Duchesne. \\

Going back to the statement we wish to prove, Theorem \ref{main3}, we see that $\Gamma$ must have a finite index subgroup $\Gamma'$ that acts on $X$ without permuting the factors. If we prove that the action of $\Gamma'$ has a fixed point in $X$, then the action of $\Gamma$ will have a finite orbit and hence a fixed point (\cite{bh}, II.2.8). Therefore, from now on, we may just assume that the action of $\Gamma$ does not permute the factors. \\

Now, $\Gamma$ acts on each factor and the action on $X$ is just the product of these actions. The action on the Euclidean factor has a fixed point due to Property (FH). Each $M_i$ and each $N_i$ is a symmetric space that belongs to the class $\mathcal{X}$, hence the action of $\Gamma$ has a fixed point in the space or in its boundary (Property ($\text{F}_{\mathcal{X}}$)). If we prove Theorem \ref{main3} in the case where $X$ is a single factor, we will find that in the general case the action on each factor has a fixed point, and therefore the global action has a fixed point, proving the theorem. Hence from now on we will assume that $X$ is either an irreducible finite-dimensional symmetric space of non-compact type, or a space $X_p(\mathbb{K})$ for some $p\geq 1$ and $\mathbb{K}\in\{\mathbb{R},\mathbb{C},\mathbb{H}\}$. \\

\subsection{Property (FH) and soluble normal subgroups}

The main tool to prove Theorem \ref{main3} will be a form of Levi decomposition for parabolic subgroups of $H=\mathrm{Isom}(X)$, where a parabolic subgroup will be expressed as a semidirect product of a soluble normal subgroup and a Levi factor of lower rank. In this subsection, we prove a lemma that will allow us to restrict the action to the Levi factor, and thus to reduce the rank of the target space.

\begin{lem} \label{soluble}
    Let $\Gamma$ be a locally compact group with Property (FH), and let $H=N\rtimes_\phi Q$ be a semidirect product of topological groups. Assume that $N$ is soluble and that every Abelian quotient that arises from its derived series is isomorphic (as a topological group) to the additive group of a Hilbert space, where the action of $Q$ (induced from $\phi$) preserves a scalar product. Then any continuous homomorphism $f:\Gamma\rightarrow H$ has image contained in a conjugate of $Q$.
\end{lem}

\begin{proof}
    We prove the lemma by induction on the solubility length of $N$. The base case will be the trivial case with solubility length zero, where $N$ is the trivial group, and $H=Q$. For the inductive step, we assume that $N$ has solubility length $n\geq 1$ and that the lemma is true in all cases with solubility length $\leq n-1$. \\

    Let $N':=[N,N]$ be the commutator subgroup of $N$, which has solubility length $n-1$, let $A=N/N'$ be the abelianization of $N$, and let $\mathrm{ab}: N\rightarrow A$ be the natural projection (with kernel $N'$). We recall that $Q$ acts on $N$ by the map $\phi: Q\rightarrow\mathrm{Aut}(N)$ that appears in the semidirect product and that associates to every $q\in Q$ the conjugation by $q$. Since the commutator subgroup $N'$ is characteristic, it is preserved by the action of $Q$, which means that the action descends to the quotient $A$. We call $\phi_{\mathrm{ab}}$ this induced action. By hypothesis, $A$ is isomorphic (as a topological group) to the additive group of a Hilbert space, and the action $\phi_{\mathrm{ab}}$ of $Q$ preserves a scalar product in this Hilbert space. \\

    We can now define an action of $\Gamma$ on $A$ as follows: for every $g\in \Gamma$ there exist unique $n(g)\in N, q(g)\in Q$ such that $f(g)=n(g)q(g)$; then we set $$F(g)(a)=\mathrm{ab}(n(g))\phi_{\mathrm{ab}}(q(g))(a)\ \ \ \forall a\in A.$$
    
    The action of $Q$ is by automorphisms, so by linear operators, while the action of $N$ is by translations. Therefore, the action of $\Gamma$ is by affine transformations, and as the action of $Q$ preserves a scalar product, the whole action of $\Gamma$ is by isometries with respect to the metric induced by this scalar product on $A$. \\
    
    Using Property (FH), we find that the action of $\Gamma$ on $A$ fixes a point $a_0$, which means that if we conjugate $f$ by an element $n_0\in N$ with $\mathrm{ab}(n_0)=a_0$ we get an action that fixes the neutral element of $A$, which means that a conjugate of the homomorphism $f$ has image contained in $N'\rtimes Q$. Now, since $N'$ has solubility length $n-1$, by inductive hypothesis this conjugate has image contained in a conjugate of $Q$, so in the end we find that $f$ has image contained in a conjugate of $Q$.
    
\end{proof}

\subsection{The induction in the finite-dimensional case}

The proof of Theorem \ref{main3} will be by induction on the rank of the target space $X$. The base case (rank 0) is trivial: a simply connected symmetric space of rank 0 is reduced to a point because otherwise it would contain a geodesic that is an embedded copy of $\mathbb{R}$. Therefore, it suffices to prove our theorem when $X$ has rank $p\geq 1$ assuming that it is true for spaces of rank $\leq p-1$. \\

As the group $\Gamma$ has Property ($\text{F}_{\mathcal{X}}$), we know that every action of $\Gamma$ on $X$ has to fix a point in $X$ or in $\partial X$. The first case is the thesis of our theorem, so from now on we will assume that we are in the second case, i.~e.~$\Gamma$ fixes $x\in\partial X$. This means that there is a continuous homomorphism $f: \Gamma\rightarrow H_x$, where $H_x<H=\mathrm{Isom}(X)$ is the parabolic subgroup defined by $H_x=\{h\in H: hx=x\}$. \\

In this subsection, we focus on the case where $X$ is finite-dimensional. The proof for this case uses standard arguments, but we include it for completeness since we did not find it in the literature. We will use the classical theory of Lie groups and symmetric spaces of non-compact type, for which a good reference is the book by Eberlein \cite{eberlein}. We use results from (2.17) of that book, which we report here using our notation. \\

At the beginning, Eberlein defines a parabolic subgroup $H_x$ of $H$ as the stabilizer of a point at infinity $x$, then defines a homomorphism $T_x: H_x\rightarrow H$ by fixing a point $p\in X$ and a tangent vector $V$ at $p$ pointing to $x$, and setting $$T_x(h)=\lim_{t\rightarrow +\infty}{e^{-tV}he^{tV}}.$$ Then he proves (2.17.5) that there is a decomposition $$H_x=N_x Z_x$$ where $N_x$ is the kernel of the homomorphism $T_x$ and $Z_x=\{h\in H: he^{tV}=e^{tV}h\ \forall t\in\mathbb{R}\}\subseteq H_x$. The subgroup $N_x$ is connected and normal in $H_x$, and the decomposition of any $h\in H_x$ as the product of an element in $N_x$ and an element in $Z_x$ is unique. It follows that $H_x$ is a semidirect product $$H_x=N_x\rtimes Z_x.$$ Eberlein then defines $\mathfrak{n}_x$ as the Lie algebra of $N_x$ and shows that it is nilpotent (2.17.14), which means that the group $N_x$ is nilpotent in the algebraic sense (\cite{hilneeb}, Theorem 11.2.5). Moreover, as shown in the proof of the Iwasawa decomposition (\cite{hilneeb}, Theorem 13.3.8), $N_x$ is simply connected. \\

Now, if we have the homomorphism $f: \Gamma\rightarrow H_x$, we can compose it with the projection $\pi: H_x\rightarrow Z_x$ and get a homomorphism $\overline{f}: \Gamma\rightarrow Z_x$. In the proof of (\cite{eberlein}, 2.17.5) it is shown that if $\gamma_{px}$ is a geodesic that passes through $p$ and points to $x$, the set $F(\gamma_{px})$ defined as the union of all geodesics parallel to $\gamma_{px}$ is preserved by the action of $Z_x$. Therefore, we have an action of $\Gamma$ on $F(\gamma_{px})$. It is shown in \cite{eberlein} (2.11.4) that $F(\gamma_{px})$ is a complete totally geodesic submanifold of $X$ and $$F(\gamma_{px})=\mathbb{R}^r\times F_S(\gamma_{px})$$ where $r\geq 1$ and $F_S(\gamma_{px})$ is a symmetric space of non-compact type. \\

As $\Gamma$ acts on $F(\gamma_{px})$, the result of De Rham \cite{derham} implies that the action splits as the product of an action on $\mathbb{R}^r$ and an action on $F_S(\gamma_{px})$. Due to Property (FH), the action on $\mathbb{R}^r$ fixes a point; hence there exists an embedded copy of $F_S(\gamma_{px})$ that is preserved by $\Gamma$ and on which $\Gamma$ acts. \\

As $\Gamma$ has Property ($\text{F}_{\mathcal{X}}$) it must fix a point in $F_S(\gamma_{px})$ or in its boundary. Since $F(\gamma_{px})$ is embedded in $X$ and since $r\geq 1$, the rank of $F_S(\gamma_{px})$ must be strictly lower than the rank of $X$, so we can use our inductive hypothesis and find that $\Gamma$ has to fix a point inside $F_S(\gamma_{px})$. \\

This means that the image of the homomorphism $\overline{f}$ is contained in a compact subgroup $Q$ of $Z_x$, and therefore the image of $f$ is contained in $N_x\rtimes Q$. Now we would like to apply Lemma \ref{soluble}, which would allow us to conclude that the image of $f$ is actually contained in a conjugate of $Q$, and therefore that the action of $\Gamma$ fixes a point in $X$, which is the thesis of Theorem \ref{main3}. \\

To verify the hypotheses of Lemma \ref{soluble}, we need to see that each Abelian quotient that appears in the derived series of the soluble group $N_x$ is isomorphic to the additive group of a Hilbert space where the action of $Q$ preserves a scalar product. We know from \cite{hilneeb} (Proposition 11.2.4) that the groups in the derived series of a connected Lie group $\mathcal{G}$ are integral subgroups (that is, they are generated by a subalgebra of the Lie algebra of $\mathcal{G}$), and that integral subgroups $\mathcal{H}$ of a simply connected soluble Lie group $\mathcal{G}$ are closed and simply connected, with $\mathcal{G}/\mathcal{H}$ isomorphic to $\mathbb{R}^{\dim{\mathcal{G}/\mathcal{H}}}$ (\cite{hilneeb}, Proposition 11.2.15). We can apply this to the Lie group $\mathcal{G}=N_x$, finding that the groups in the derived series are closed and simply connected, and each quotient is isomorphic to $\mathbb{R}^n$ for some $n$. We also know that the groups in the derived series are characteristic subgroups of $N_x$ (meaning that they are preserved by any automorphism of $N_x$), and hence the action of $Q$ given by the semidirect product $N_x\rtimes Q$ descends to an action on each quotient. Furthermore, since $Q$ is compact, each quotient has a scalar product that is preserved by this action of $Q$: it suffices to start from any scalar product and integrate it over $Q$ to obtain a $Q$-invariant one. This completes the proof. \\

\subsection{The infinite-dimensional case}

We are left with showing the inductive step for the infinite-dimensional case, where we have $X=X_p(\mathbb{K})$ for some $p\geq 1$ and $\mathbb{K}\in\{\mathbb{R},\mathbb{C},\mathbb{H}\}$. Since $\Gamma$ acts continuously by isometries on $X$, it has a continuous homomorphism to the group $H=\mathrm{PO}_\mathbb{K}(p,\infty)$, where $\mathrm{PO}_\mathbb{R}(p,\infty)=\mathrm{PO}(p,\infty)$, $\mathrm{PO}_\mathbb{C}(p,\infty)=\mathrm{PU}(p,\infty)$, and $\mathrm{PO}_\mathbb{H}(p,\infty)=\mathrm{PSp}(p,\infty)$. In fact, it is shown in \cite{duc2023} (Theorem 3.3) that $\mathrm{Isom}(X_p(\mathbb{K}))=\mathrm{PO}_\mathbb{K}(p,\infty)$, except in the case $\mathbb{K}=\mathbb{C}$ where $\mathrm{PU}(p,\infty)$ is a subgroup of index two in $\mathrm{Isom}(X_p(\mathbb{C}))$. In that case, we have a subgroup $\Gamma'<\Gamma$ of index $\leq 2$ with a homomorphism to $\mathrm{PU}(p,\infty)$, and if we show that $\Gamma'$ fixes a point in $X$, then this implies that $\Gamma$ has a finite orbit and hence fixes a point. \\

Therefore, we consider a continuous homomorphism $f: \Gamma\rightarrow H=\mathrm{PO}_\mathbb{K}(p,\infty)$, that is, a continuous action of $\Gamma$ by projective transformations on a $\mathbb{K}$-vector space $V$ of countable dimension that preserves a nondegenerate sesquilinear form of index $p$. \\

Since $\Gamma$ has Property ($\text{F}_{\mathcal{X}}$), we can assume that the action on $X=X_p(\mathbb{K})$ associated with $f$ fixes a point $x\in\partial X$, therefore, the image of $f$ is contained in $H_x=\{h\in H: hx=x\}$. \\

In the course of the proof, we will consider matrices associated to elements of $\mathrm{PO}_\mathbb{K}(p,\infty)$ or other projective groups. This means that we are writing the explicit matrix form of a representative of the element in the corresponding matrix group (in this case, $\mathrm{O}_\mathbb{K}(p,\infty)$). Moreover, in the case $\mathbb{K}=\mathbb{H}$, we use the convention for which the multiplication by a scalar in the vector space is on the right, so that the action is on the left.  \\

We can choose a basis $(e_1,e_2,\ldots)$ of $V$ such that the sesquilinear form is expressed by the infinite matrix $$\large \left(\begin{array}{c|c|c}
    0 & Id_p & 0 \\
    \hline
    Id_p & 0 & 0 \\
    \hline
    0 & 0 & Id_\infty
\end{array}\right)$$ where $Id_p$ is the identity $p\times p$ matrix and $Id_\infty$ is the identity infinite matrix. \\

Each point at infinity of $X_p(\mathbb{K})$ is associated with a flag of isotropic subspaces of the vector space $V$: this was shown in \cite{duc2013} (Proposition 6.1) in the real case, but is true in all cases, as observed in \cite{duc2023} (as a consequence of Remark 3.1 in \cite{duc2023}). \\

We consider the flag associated with the point at infinity $x$ (which is fixed by the action of $\Gamma$). Let $q\leq p$ be the dimension of the maximal subspace in the flag. We may choose the basis in such a way that this maximal subspace is exactly $\mathrm{Span}\{e_1,\ldots,e_q\}$. Then we can change the elements $(e_{q+1},e_{q+2},\ldots)$ of the basis in such a way that the sesquilinear form is now expressed by the matrix $$\large\left(\begin{array}{c|c|c}
	0 & Id_q & 0 \\
	\hline
	Id_q & 0 & 0 \\
	\hline
	0 & 0 & \Large J
\end{array}\right),\ \ \ J\normalsize=\left(\begin{array}{c|c}
	-Id_{p-q} & 0 \\
	\hline
	    0 & Id_\infty
\end{array}\right).$$ \\

Now we consider the matrix associated with a transformation in $H_x$ with respect to this new basis and find some restrictions on its structure. We see it as a block matrix similar to the one expressing the sesquilinear form, i.~e.~with two blocks of dimension $q$ and one block of infinite dimension, both horizontally and vertically. This means that we are splitting the vector space $V$ into a direct sum $V_1\oplus V_2\oplus V_3$ where $V_1$ and $V_2$ have dimension $q$ and $V_3$ has infinite dimension, and each block of the matrix corresponds to a continuous linear map from $V_i$ to $V_j$ for some $i,j\in\{1,2,3\}$. \\

First of all, we note that to preserve $V_1$ the two bottom blocks in the left column must vanish. For the upper left block, we generally have a block-upper triangular structure depending on the flag associated with $x$, but we are not going to need that structure in our proof, so we may just assume that it belongs to some subgroup $S<\mathrm{GL}_q(\mathbb{K})$ (possibly the whole of $\mathrm{GL}_q(\mathbb{K})$). We get the form: $$\large \left(\begin{array}{c|c|c}
    M & Y & Z \\
    \hline
    0 & L & W \\
    \hline
    0 & B & R
\end{array}\right).$$

In order to preserve the sesquilinear form, we must have $$\left(\begin{array}{c|c|c}
    M^T & 0 & 0 \\
    \hline
    Y^T & L^T & B^T \\
    \hline
    Z^T & W^T & R^T
\end{array}\right)\left(\begin{array}{c|c|c}
    0 & Id_q & 0 \\
    \hline
    Id_q & 0 & 0 \\
    \hline
    0 & 0 & J
\end{array}\right)\left(\begin{array}{c|c|c}
    M & Y & Z \\
    \hline
    0 & L & W \\
    \hline
    0 & B & R
\end{array}\right)=\left(\begin{array}{c|c|c}
    0 & Id_q & 0 \\
    \hline
    Id_q & 0 & 0 \\
    \hline
    0 & 0 & J
\end{array}\right)$$ which is equivalent to $$\left(\begin{array}{c|c|c}
    0 & M^T L & M^T W \\
    \hline
    L^T M & L^T Y + Y^T L + B^T J B & L^T Z + Y^T W + B^T J R \\
    \hline
    W^T M & W^T Y + Z^T L + R^T J B & W^T Z + Z^T W + R^T J R
\end{array}\right)=\left(\begin{array}{c|c|c}
    0 & Id_q & 0 \\
    \hline
    Id_q & 0 & 0 \\
    \hline
    0 & 0 & J
\end{array}\right).$$

From the first row we immediately get $L=(M^T)^{-1}$ and $W=0$, so we might simplify to $$\left(\begin{array}{c|c|c}
    0 & Id_q & 0 \\
    \hline
    Id_q & M^{-1} Y + (M^{-1} Y)^T + B^T J B & M^{-1} Z + B^T J R \\
    \hline
    0 & (M^{-1} Z)^T + R^T J B & R^T J R
\end{array}\right)=\left(\begin{array}{c|c|c}
    0 & Id_q & 0 \\
    \hline
    Id_q & 0 & 0 \\
    \hline
    0 & 0 & J
\end{array}\right).$$

From $R^T J R = J$ we see that $R\in \mathrm{O}_\mathbb{K}(p-q,\infty)$. From $M^{-1} Z + B^T J R = 0$ we get $Z = - M B^T J R$. Therefore, the matrix of our transformation has the form $$\large \left(\begin{array}{c|c|c}
    M & Y & -M B^T J R \\
    \hline
    0 & (M^T)^{-1} & 0 \\
    \hline
    0 & B & R
\end{array}\right)$$ with $R\in\mathrm{O}_\mathbb{K}(p-q,\infty)$ and $Y$ satisfying $M^{-1} Y + (M^{-1} Y)^T = - B^T J B$. \\

Our next step will be to decompose the group $H_x$ as a semidirect product, using the explicit matrix form. We can project $H_x$ to the (projective) group of block-diagonal matrices, with block dimensions $q,q,\infty$, using the map $$\left(\begin{array}{c|c|c}
    M & Y & -M B^T J R \\
    \hline
    0 & (M^T)^{-1} & 0 \\
    \hline
    0 & B & R
\end{array}\right)\mapsto\left(\begin{array}{c|c|c}
    M & 0 & 0 \\
    \hline
    0 & (M^T)^{-1} & 0 \\
    \hline
    0 & 0 & R
\end{array}\right).$$

It is immediate to see that this is a well-defined group homomorphism. As $M\in S<GL_q(\mathbb{K})$ and $R\in \mathrm{O}_\mathbb{K}(p-q,\infty)$, this gives a projection $H_x\rightarrow \mathrm{P}(S\times\mathrm{O}_\mathbb{K}(p-q,\infty))$. \\

The kernel $N$ of this projection is the space of matrices having the form $$\large \left(\begin{array}{c|c|c}
    Id_q & Y & -B^T J \\
    \hline
    0 & Id_q & 0 \\
    \hline
    0 & B & Id_\infty
\end{array}\right)$$ with $Y + Y^T = - B^T J B$. We will show that this group is soluble with solubility length two and that it satisfies the assumptions of Lemma \ref{soluble}. \\

The map sending a matrix as above to the block $B$ determines a homomorphism $N\rightarrow A$ where $A$ is the additive group of $q\times\infty$ matrices (i.~e.~linear maps from the $q$-dimensional vector space $V_2$ to the infinite-dimensional Hilbert space $V_3$), which is Abelian and isomorphic with the direct sum of $q$ Hilbert spaces of countable dimension, which is isomorphic with a single Hilbert space of countable dimension. The kernel of the map is the set of matrices in $N$ with $B=0$, which is isomorphic with the additive group of $q\times q$ matrices $Y$ that satisfy $Y+Y^T=0$, that is, the additive group of skew-symmetric matrices, isomorphic with $\mathbb{K}^{q(q-1)/2}$. \\

Therefore, we have a semidirect product decomposition $$H_x=N\rtimes\mathrm{P}(S\times \mathrm{O}_\mathbb{K}(p-q,\infty))$$ where $N$ is soluble. \\

What we need to show for our theorem is that any continuous homomorphism $f: \Gamma\rightarrow H_x$ has image contained in the stabilizer of a point in the symmetric space $X$. Consider the homomorphisms $f_S: \Gamma\rightarrow\mathrm{P}S$ and $f_O: \Gamma\rightarrow \mathrm{PO}_\mathbb{K}(p-q,\infty)$ defined by $f_S=\pi_S\circ f$ and $f_O=\pi_O\circ f$, where $\pi_S$ and $\pi_O$ are the natural projections of $H_x$ onto $\mathrm{P}S$ and $\mathrm{PO}_\mathbb{K}(p-q,\infty)$, respectively. \\

The homomorphism $f_S$ has image in $\mathrm{P}S<\mathrm{PGL}_q(\mathbb{K})$. As $\mathrm{PGL}_q(\mathbb{K})$ is a Lie group of rank $q-1<p$, we can use the inductive hypothesis of Theorem \ref{main3} and see that the image of $f_S$ fixes a point in the associated symmetric space. Up to conjugating the homomorphism (which is the same as changing the basis of $V_1$ and $V_2$ that we consider in our matrix decomposition), we can assume that the image of $f_S$ is contained in $\mathrm{PO}_\mathbb{K}(q)$. \\

Similarly, we see that the image of $f_O$ is contained in a subgroup that fixes a point in the infinite-dimensional symmetric space $X_\mathbb{K}(p-q,\infty)$. Up to conjugating, we may therefore assume that the image of $f_O$ is contained in $\mathrm{P}(\mathrm{O}_\mathbb{K}(p-q)\times \mathrm{O}_\mathbb{K}(\infty))$, meaning that every element is block-diagonal for the decomposition of $V_3$ used in the definition of the matrix $J$, and that the two nonzero blocks that appear are in $\mathrm{O}_\mathbb{K}(p-q)$ and $\mathrm{O}_\mathbb{K}(\infty)$. \\

Now we know that the image of $f$ is contained in $N\rtimes Q$ with $Q=\mathrm{P}(\mathrm{O}_\mathbb{K}(q)\times\mathrm{O}_\mathbb{K}(p-q)\times\mathrm{O}_\mathbb{K}(\infty))$. We wish to apply Lemma \ref{soluble} to find that it is actually contained in a conjugate of $Q$, which would prove our theorem. We only need to check the hypothesis that the conjugation of elements in $N$ by elements of $Q$ preserves a Hilbert scalar product in the two Abelian groups that appear in the soluble group structure, namely the additive group $A$ of maps $V_2\rightarrow V_3$ (the block $B$ in the decomposition) and the additive group of skew-symmetric matrices $Y$. \\

We show that the standard scalar product in these spaces is preserved. The norm associated to this scalar product can be regarded, in both cases, as the Hilbert sum of the squares of all coefficients of the associated matrix (in the case of skew-symmetric matrices, summing over the whole matrix instead of just on the upper-triangular part multiplies all the norms by $2$). The conjugation of a generic element in $N$ by a generic element in $S\times \mathrm{O}_\mathbb{K}(p-q,\infty)$ gives
\begin{align*}
\left(\begin{array}{c|c|c}
    M & 0 & 0 \\
    \hline
    0 & (M^T)^{-1} & 0 \\
    \hline
    0 & 0 & R
\end{array}\right)&\left(\begin{array}{c|c|c}
    Id_q & Y & -B^T J \\
    \hline
    0 & Id_q & 0 \\
    \hline
    0 & B & Id_\infty
\end{array}\right)\left(\begin{array}{c|c|c}
    M^{-1} & 0 & 0 \\
    \hline
    0 & M^T & 0 \\
    \hline
    0 & 0 & R^{-1}
\end{array}\right)= \\
&=\left(\begin{array}{c|c|c}
    Id_q & M Y M^T & -M B^T J R^{-1} \\
    \hline
    0 & Id_q & 0 \\
    \hline
    0 & R B M^T & Id_\infty
\end{array}\right).
\end{align*}

In our case $M\in O_{\mathbb{K}}(q)$, and this also implies $M^T\in O_{\mathbb{K}}(q)$. When a matrix $Y$ is multiplied by $M^T$ on the right, every row is transformed with an orthogonal transformation and therefore its norm stays the same. Hence, the norm of the whole matrix remains the same. Similarly, when multiplying by $M$ on the left, every column is transformed orthogonally and the total norm remains the same. \\

As for the transformation $B\mapsto R B M^T$, the same remark applies for the right multiplication by $M^T$, and also for the left multiplication by $R$ since $R\in\mathrm{O}_\mathbb{K}(p-q)\times\mathrm{O}_\mathbb{K}(\infty)$ and hence it preserves the standard scalar product on $V_3$. This concludes the proof of Theorem \ref{main3}. \\

\subsection{Conclusion}

We are now ready to prove Theorem \ref{main}, which implies Theorem \ref{Main} and Theorem \ref{Main2}. \\

Let $G=\mathrm{SL}_3(\mathbb{F})$ or $G=\mathrm{Sp}_4(\mathbb{F})$, with $\mathbb{F}$ a non-archimedean local field; in the case $G=\mathrm{Sp}_4(\mathbb{F})$ we assume that $\mathbb{F}$ has characteristic zero and its residue field $k$ has at least three elements. Let $\Gamma$ be a torsion-free cocompact lattice of $G$, which exists by (\cite{borelharder}, Theorem A, Theorem 3.3) and (\cite{garland}, Theorem 2.7). By Theorem \ref{main2}, $\Gamma$ has Property ($\text{F}_{\mathcal{X}}$). \\

The group $G$ has Kazhdan's Property (T) (\cite{kazhdanT}, 1.4, 1.5), and since this property is inherited by lattices (\cite{kazhdanT}, Theorem 1.7.1), the lattice $\Gamma$ also has Kazhdan's Property (T). By the Delorme-Guichardet Theorem (\cite{kazhdanT}, Theorem 2.12.4), this means that these groups have Property (FH). \\

By Theorem \ref{main3}, any continuous action by isometries of $\Gamma$ on a space $X\in\mathcal{X}$ has a fixed point in $X$. This proves Theorem \ref{main}. \\

\section*{Acknowledgments}

This paper is part of my Ph.D. project under the supervision of Prof. Nicolas Monod at EPFL (Lausanne, Switzerland). I am grateful to Nicolas for proposing this project, for giving me valuable advice and feedback, and for sharing with me deep insights of the theory. I am also grateful to Bruno Duchesne, Pierre Py, and the reviewers for their useful comments and suggestions on this paper, and in particular for the suggestion to generalize the original version of the theorem. \\

\appendix
\section{Appendix}

In this appendix we prove that the spaces $X_p(\mathbb{K})$, defined at the beginning of Section \ref{morefamily}, are irreducible metric spaces, meaning that they are not isometric to the product of two metric spaces not reduced to a point. \\

Set $X=X_p(\mathbb{K})$. We suppose by contradiction that $X=X_1\times X_2$ for some metric spaces $X_1,X_2$ not reduced to a point. First, we note that $X_1,X_2$ must be geodesic spaces, since $X$ is a geodesic space and geodesics in a product project to geodesics in factors (\cite{bh}, I.5.3). Moreover, geodesics in $X_1,X_2$ must extend indefinitely, since geodesics in $X$ extend indefinitely. This implies that the boundary at infinity of $X_1$ and $X_2$ is non-empty. We also note that $X_1,X_2$ must be CAT(0) spaces, as they are convex subsets of the CAT(0) space $X$, and that they must be complete, as their product is complete. \\

From now on, we will assume that $X$ has rank at least $2$. In fact, since both $X_1$ and $X_2$ contain an infinite geodesic, $X$ must contain an isometrically embedded copy of $\mathbb{R}^2$. \\

It is proven in \cite{bh} (II.9) that the Tits boundary $\partial_T(X_1\times X_2)$ of the product of two complete CAT(0) spaces is the spherical join of the Tits boundaries of the factors. Therefore, it suffices to prove:

\begin{prp} \label{nojoin}
    For $X=X_p(\mathbb{K})$, the Tits boundary $\partial_T X$ is not the spherical join of two non-empty metric spaces. \\
\end{prp}

The Tits boundary of $X_p(\mathbb{K})$ was studied by Duchesne \cite{duc2013}: it is a thick spherical building, and all finite configurations of simplices are the same that appear in the spherical building of the finite-dimensional analogue $\mathrm{O}_\mathbb{K}(p,q)$ for sufficiently large $q$. Details on the structure of these buildings can be found in the book by Abramenko and Brown \cite{buildings} (6.7). We show the following:

\begin{lem}
    In the spherical building of $\mathrm{O}_\mathbb{K}(p,q)$, every simplex has diameter strictly smaller than $\pi/2$.
\end{lem}

\begin{proof}
    We refer to \cite{buildings} for details on the construction of the building. Here we only describe the structure of the fundamental apartment, which allows us to see that every simplex has diameter $<\pi/2$. \\

    The building is the flag complex of totally isotropic subspaces of a $(p+q)$-dimensional $\mathbb{K}$-vector space equipped with a sesquilinear form $\langle\cdot,\cdot\rangle$ of index $p$ (we assume $p<q$). The fundamental apartment is associated with a $(2p)$-tuple of linearly independent vectors $(e_1,\ldots,e_p,f_1,\ldots,f_p)$ such that $\langle e_i,e_j\rangle=\langle f_i,f_j\rangle=0$ for $1\leq i,j\leq p$, $\langle e_i,f_i\rangle=1$ for $1\leq i\leq p$, and $\langle e_i,f_j\rangle=0$ for $i\neq j$. The apartment is isometric to the sphere $S^{p-1}$ and is the flag complex of totally isotropic subspaces spanned by elements of the above $(2p)$-tuple. \\

    Now we describe the structure of the simplices in the sphere, which is forced by symmetry. We can choose coordinates such that the sphere is the unit sphere in $\mathbb{R}^p$, for $1\leq i\leq p$ the subspace $\mathrm{Span}\{e_i\}$ is associated with the point with $i$-th coordinate equal to $1$, and the subspace $\mathrm{Span}\{f_i\}$ is associated with the point with $i$-th coordinate equal to $-1$. These points determine a subdivision of the sphere into $2^p$ \textit{super-simplices}: each of them is the intersection of the sphere with the subset of $\mathbb{R}^p$ resulting from a specific choice of the signs of the $p$ coordinates. Each super-simplex corresponds to a maximal isotropic subspace and is then subdivided barycentrically into $p!$ individual simplices. Each subset $A\subset\{e_1,\ldots,e_p,f_1,\ldots,f_p\}$ with $\mathrm{Span}(A)$ totally isotropic is associated with the barycenter of the set of points associated with elements of $A$. \\

    Clearly, all super-simplices are isometric and all individual simplices are isometric. Let $\mathcal{S}$ be a super-simplex, without loss of generality, $$\mathcal{S}=\{(x_1,\ldots,x_p)\in S^{p-1}: x_i\geq 0, 1\leq i\leq p\}.$$
    We now describe the individual simplices that make up $\mathcal{S}$. For each permutation $\sigma$ of $\{1,\ldots,p\}$, we have the simplex $$s_\sigma=\{(x_1,\ldots,x_p)\in\mathcal{S}: x_{\sigma(1)}\geq x_{\sigma(2)}\geq\cdots\geq x_{\sigma(p)}\}.$$ This corresponds to the flag $$\mathrm{Span}\{e_{\sigma(1)}\}\subset\mathrm{Span}\{e_{\sigma(1)},e_{\sigma(2)}\}\subset\cdots\subset\mathrm{Span}\{e_{\sigma(1)},\ldots,e_{\sigma(p)}\}.$$

    We assume without loss of generality $\sigma=id$ and show that the simplex $s_{id}$ has diameter $<\pi/2$. Let $x\in S^{p-1}$. The set of points in $S^{p-1}$ at distance $\geq\pi/2$ from $x$ can be defined by $S^{p-1}\cap\{y\in\mathbb{R}^p: \langle x,y\rangle\leq 0\}$, where we use the standard scalar product in $\mathbb{R}^p$. If we assume that there exist $x,y\in s_{id}$ with $d(x,y)\geq\pi/2$, they must satisfy $\langle x,y\rangle\leq 0$. As all their coordinates are non-negative and their first coordinate is at least $\sqrt{1/p}$, we have $\langle x,y\rangle\geq 1/p$, a contradiction. Hence, the diameter of $s_{id}$, and similarly the diameter of any simplex in the sphere, is strictly smaller than $\pi/2$. \\
\end{proof}

Now we can deduce Proposition \ref{nojoin} from the following general result:

\begin{thm}
    Let $\Sigma$ be a thick spherical building in which every simplex has diameter strictly smaller than $\pi/2$. Then $\Sigma$ is not the spherical join of two non-empty metric spaces.
\end{thm}

\begin{proof}
    Suppose by contradiction that $\Sigma$ is the spherical join $A_1*A_2$. Then $A_1,A_2$ are non-empty subsets of $\Sigma$ with the following properties (deduced from the construction of the spherical join, see \cite{bh}, I.5.13):
    \begin{enumerate}[label=(\alph*)]
        \item for all $x\in A_1, y\in A_2$ we have $d(x,y)=\pi/2$;
        \item for all $z\in\Sigma$ there exist $x\in A_1,y\in A_2$ such that $z$ belongs to the unique geodesic segment between $x$ and $y$; if $z\not\in A_1\cup A_2$, these $x,y$ are the unique points in $A_1,A_2$ at minimal distance from $z$.
    \end{enumerate}
    Moreover, since $\Sigma$ is a CAT(1) space (\cite{bh}, II.10A.4), $A_1$ and $A_2$ must also be CAT(1) spaces (\cite{bh}, II.3.15). In particular they are $\pi$-geodesic. \\

    Let $x\in A_1, y\in A_2$. By (a) they are at distance $\pi/2$, thus there exists a unique geodesic segment between them. Since $\Sigma$ is a spherical building, this segment has to be part of a geodesic loop $\gamma$ of length $2\pi$. Let $x'$ be the point on this loop that is antipodal to $x$. Let $z_1\in\gamma$ such that $d(y,z_1)=\pi/2-\varepsilon$ and $d(z_1,x')=\varepsilon$, where $\varepsilon$ is small enough so that the geodesic from $y$ to $z_1$ extends in a unique way at least up to $x'$. Let $x_1\in A_1,y_1\in A_2$ be the points obtained by applying property (b) on $z_1$. We have $d(x,y_1)=\pi/2$ by property (a), hence $d(y_1,z_1)\geq|d(x,z_1)-d(x,y_1)|=\pi/2-\varepsilon$ by the triangle inequality. Since $y_1$ is the unique point in $A_2$ at minimal distance from $z_1$, and since $d(y,z_1)=\pi/2-\varepsilon$, we must have $y_1=y$, and $x_1=x'$ since it must be at distance $\epsilon$ from $z_1$ on a geodesic extending the one coming from $y$ and the geodesic $\gamma$ is the unique extension at least up to the point $x'$. Therefore $x'\in A_1$. \\

    Now, since simplices in $\Sigma$ have diameter strictly smaller than $\pi/2$, and since the building is thick, there exists a point $r$ in the interior of the geodesic segment between $y$ and $x'$ where the geodesic ramifies, meaning that the extension beyond $r$ of the geodesic segment from $y$ to $r$ is not unique. The point $z_1$ is of course on the segment between $r$ and $x'$. The non-uniqueness of the geodesic extension beyond $r$ means that there exists another geodesic $\widetilde{\gamma}$ that contains the points $x,y,r$ but not the point $x'$. It will contain instead another point $x''$ antipodal to $x$. We can choose $z_2$ on the geodesic segment between $r$ and $x''$, close enough to $x''$ so that the geodesic segment from $r$ to $z_2$ extends uniquely at least up to $x''$, and repeat the same proof as before to show that $x''\in A_1$. \\
    
    Now we have $d(x',x'')\leq d(x',r)+d(r,x'')=\pi-2d(y,r)<\pi$, hence since $A_1$ is $\pi$-geodesic the geodesic segment between $x'$ and $x''$ is contained in $A_1$. Let $m$ be the midpoint of this segment. From the CAT(1) inequality applied on the triangle with vertices $r,x',x''$ we get $d(r,m)<\pi/2-d(y,r)$. Therefore $d(y,m)\leq d(y,r)+d(r,m)<\pi/2$, which is a contradiction since $y\in A_2$ and $m\in A_1$. \\
\end{proof}

\printbibliography

Federico Viola, EPFL, Lausanne, CH-1015, Switzerland.

\textit{E-mail address: federico.viola@epfl.ch}

\end{document}